\documentclass[11pt]{amsart}
\usepackage[english]{babel}

\usepackage{latexsym}
\usepackage{amssymb}
\usepackage{amsmath}
\usepackage{amsfonts}
\usepackage[mathscr]{eucal}
\usepackage{color}

\usepackage{amscd}

\usepackage{graphicx}
\usepackage{epsfig}
\usepackage{relsize}

\usepackage{fullpage}

\usepackage[table,xcdraw]{xcolor}
\usepackage{longtable}
\usepackage{rotating}
\usepackage{array}
\usepackage{multicol}
\usepackage{multirow}
\usepackage{makecell}
\usepackage{mathtools}
\usepackage{booktabs}

\usepackage{amsthm}
\usepackage[all,cmtip]{xy}
\usepackage[pdfstartview=FitH]{hyperref}
\usepackage{comment}
\usepackage{psfrag}
\usepackage{verbatim}
\usepackage{hyperref}
\hypersetup{
colorlinks=true,
linkcolor=blue
}
\usepackage{pifont}
\usepackage{caption}
\usepackage{csquotes}

\usepackage{soul}
 \usepackage[normalem]{ulem}

\usepackage{tikz}
\usepackage{pgfplots}
\pgfplotsset{compat=1.15}
\usepackage{mathrsfs}
\usetikzlibrary{arrows}

\usepackage[style=numeric,backend=biber]{biblatex}

\newtheorem{thm}{Theorem}[section]
\newtheorem*{thm*}{Theorem}
\newtheorem{cor}[thm]{Corollary}	
\newtheorem*{cor*}{Corollary}		

\newtheorem{prop}[thm]{Proposition}

\theoremstyle{definition}
\newtheorem{definition}[thm]{Definition} 	
 
\newtheorem*{definition*}{Definition}    	

\newtheorem{example}[thm]{Example}

\numberwithin{equation}{section}

\DeclareMathOperator{\dist}{dist}

\DeclareMathOperator{\diam}{diam}
\DeclareMathOperator{\vol}{vol}

\DeclareMathOperator{\dil}{dil}
\DeclareMathOperator{\inj}{inj}
\DeclareMathOperator{\reach}{reach}
\DeclareMathOperator{\unp}{Unp}

\DeclareMathOperator{\Cyl}{Cyl}

\DeclareMathOperator{\fillrad}{FillRad}

\newcommand{\norm}[1]{\left\Vert#1\right\Vert}

\newcommand{\RR}{\mathbb{R}}

\newcommand{\eps}{\varepsilon}

\newcommand{\CP}{\mathbb{C}P}
\newcommand{\FF}{\mathbb{F}}

\newcommand{\sphere}{\mathrm{\mathbb{S}}}

\newcommand{\tildef}{\widetilde{f}}

\newcommand{\ZZ}{\mathbb{Z}}

\newcommand{\Linf}{L^\infty(M)}

\begin{document}



\title[Upper and lower bounds on the filling radius]{Upper and lower bounds on the filling radius \\ 
}


\author[M.~Cuerno]{Manuel Cuerno $^{\ast}$}

\author[L.~Guijarro]{Luis Guijarro$^{\ast\ast}$}


\thanks{$^*$Supported in part by the FPI Graduate Research Grant PRE2018-084109, and by research grants  
	 MTM2017-‐85934-‐C3-‐2-‐P 
from the Ministerio de Econom\'ia y Competitividad de Espa\~{na} (MINECO)} 

\thanks{$^{\ast\ast}$Supported by research grants  MTM2017-‐85934-‐C3-‐2-‐P from the  Ministerio de Econom\'{\i}a y Competitividad de Espa\~{na} (MINECO), and by ICMAT Severo Ochoa project SEV-‐2015-‐0554 (MINECO)}


\address[M.~Cuerno]{Department of Mathematics, Universidad Aut\'onoma de Madrid and ICMAT CSIC-UAM-UC3M, Spain}
\email{manuel.mellado@uam.es}

\address[L.~Guijarro]{Department of Mathematics, Universidad Aut\'onoma de Madrid and ICMAT CSIC-UAM-UC3M, Spain}
\email{luis.guijarro@uam.es}


\date{\today}




\begin{abstract}
We give a curvature dependent lower bound for the filling radius of all closed Riemannian manifolds as well as an upper one 
for manifolds  which are the total space of a Riemannian submersion. The latter applies also to the case of submetries. We also see that the  reach (in the sense of Federer) of the image of the Kuratowski embedding vanishes, and we finish by giving  some inequalities involving the $k-$intermediate filling radius.
\end{abstract}
\setcounter{tocdepth}{1}

\maketitle





	\section{Introduction}
	
	Let $M^n$ be a closed connected $n$-manifold with a Radon measure $\mu$ and a distance  $\dist_M=\dist(p,q)$. The functions  $\dist_p(\cdot)=\dist(p,\cdot)$ lie in $L^{\infty}(M)$, the space  of real bounded functions on $M$ equipped with the sup-norm $\norm{\cdot}_{L^{\infty}(M)}$. The canonical inclusion 
	\[
		\Phi:M\to L^{\infty}(M), \qquad 
		p\to\dist_p,
	\]
	is  an isometric embedding known as the \textit{Kuratowski embedding}, since 
	\[
	\dist_{M}(p,q)=\norm{\dist_p-\dist_{q}}_{L^{\infty}(M)}
	\]
by the triangle inequality. We will identify $M$ with $\Phi(M)$ in what follows, hoping this will not create confusion. Given $r>0$, we will denote by $U_r(M)$ the $r$--neighbourhood of $M$ in $\Linf$, and by $\iota_r:M\hookrightarrow U_r(M)$ the inclusion.

	For a given coefficient group $\FF$, consider the homomorphism induced in $n$-homology by the inclusion 
	\[
	\iota_{r,*}:\mathrm{H}_n(M,\FF)\to\mathrm{H}_n(U_r(M),\FF).
	\]

In what follows, $M$ will have a Riemannian metric inducing the distance in $M$ and a Riemannian volume. We will consider $\FF=\ZZ$ when $M$ is orientable, and $\FF=\ZZ_2$ otherwise. 

	\begin{definition}[Filling Radius \cite{gromov}]
		The \emph{filling radius} of $M$, denoted by $\fillrad (M)$, is the infimum of those $r>0$ for which $\iota_{r,*}([M])=0$, where $[M]$ is the fundamental class of $M$.
 	\end{definition}
 	
%

Gromov introduced the filling radius in \cite{gromov} as a tool to obtain his systolic inequality, proving along the way, that $\fillrad(M)$ is bounded above by the $n$-th root of the volume times a constant depending only on the dimension of the manifold (see also \cite{guth}).
Upper bounds of the filling radius in terms of the diameter were obtained by Katz in \cite{katz}; Wilhelm studied it for manifolds with positive sectional curvature in  \cite{wilhelm}; see also \cite{yokota} for the extension to Alexandrov spaces. Later works include \cite{liu}, \cite{guth}, \cite{nabu}, and \cite{lim2020vietoris} just to include a few, 
but this should not be considered an exhaustive list. 

In general, it is difficult to compute the exact value of the filling radius for a particular Riemannian manifold. Its exact value has been computed in very few occasions (see \cite{katz}, \cite{wilhelm}, \cite{katz_CP2}), so it is useful to get bounds for it.  
In this direction, we start by obtaining a lower bound on the filling radius of a manifold that is valid for every closed Riemannian manifold. The same result, although stated in terms of the convexity radius, appears in \cite[pp.282--284]{petersen}; nonetheless, we consider our proof to have some interest, since it could help to understanding the change in the topology of tubular neighbourhoods of the Kuratowski embedding in the spirit of \cite{katzneig}. The reader can also consult \cite{liu}, although our estimate improves his by a factor of $\pi$ and removes an extra condition on the injectivity radius.

\begin{thm}
\label{thm:lower_bound}
     Let $M$ be a closed Riemannian manifold with injectivity radius $\inj M$ and with sectional curvature $\sec\leq \Delta$, where $\Delta\geq 0$. Then 
     \begin{equation}
     \label{eq:fillrad_ineq}
     \fillrad(M)\geq \frac{1}{4} \min\left\{\,\inj M\,,\,\frac{\pi}{\sqrt{\Delta}}\,\right\},
     \end{equation}
     where $\pi/\sqrt{\Delta}$ is understood as $\infty$ whenever $\Delta= 0$. 
 \end{thm}

This has a immediate consequence.  
  \begin{cor}    
     Let $M$ be a closed Riemannian manifold. Then $\fillrad(M)$ is strictly positive.
 \end{cor}
As in \cite{petersen}, we should mention that the lower bound in Theorem \ref{thm:lower_bound} is optimal, since for the round sphere, 
\[
\fillrad\sphere^n=\frac{1}{2}\arccos\left(-\frac{1}{n+1}\right)\to \pi/4 \text{ when }n\to \infty.
\]
 
In \cite{katz}, Katz obtained upper bounds on the filling radius in terms of the diameter of the manifold. We show how his ideas can be extended to the case when $M$ is the total space of a Riemannian submersion.

	\begin{thm}
	\label{thm:submersion}
		Let $\pi:M\to B$ be a Riemannian submersion with $\dim M>\dim B$. Then 
		\begin{equation}
		\label{eq:fillrad_submersion}
		    \fillrad(M)\leq\frac12\max_{b\in B}\{\diam\pi^{-1}(b)\},
		\end{equation}
		where the diameter of each fiber is considered in the extrinsic metric. 
		
	\end{thm}
	
Our proof admits an extension to the case when the Riemannian submersion is replaced by a \emph{submetry} (see Section \ref{sec:submetry} for definitions); this has the advantage of including  quotient maps corresponding to isometric group actions, as well as manifolds $M$ admitting singular Riemannian foliations with closed leaves. 

By combining both bounds, we obtain some interesting inequalities for arbitrary Riemannian submersions:

\begin{cor}
    Let $\pi:M\to B$ a Riemannian submersion between closed manifolds. Then
\[
\frac12\min\left\{\,\inj M\,,\,\frac{\pi}{\sqrt{\Delta}}\,\right\}\leq
\max_{b\in B} \, \{\diam\pi^{-1}(b)\},
\]
where $\Delta$ is an upper positive bound for the sectional curvature of $M$.  
\end{cor}

The bound in Theorem \ref{thm:submersion} is sometimes better than the upper bound of $\diam(M)/3$ obtained by Katz in \cite{katz}; for instance, if $(\sphere^{2n+1}, g_a)$ are Berger spheres with $a\to 0$, then $\diam(\sphere^{2n+1}, g_a)\to \diam\CP^n=\pi/2$, while the diameters of the fibers of $\pi:(\sphere^{2n+1}, g_a)\to \CP^n$ tends to zero with $a\to 0$. 
Another example is presented in Section \ref{sec:submetry}.

We also show that the reach of the Kuratowski embedding is always zero, thus precluding the possibility of defining any metric projection on $M$ of $L^\infty$-neighbourhoods of $M$, even locally around a point. 

\begin{thm}
\label{thm:reach}
	Let $M^n$ be a closed Riemannian manifold. For every $p\in M$,
	\[
	\reach(p,M\subset L^{\infty}(M))=0.
	\]
\end{thm}
The reader can see the definitions, as well as the proof, in Section \ref{sec:reach}.

To finish the paper, we consider the intermediate filling radius introduced in \cite{lim2020vietoris} and obtain some bounds similar to Theorem \ref{thm:lower_bound} and \cite[Theorem 1.1]{liu}.  

\emph{Acknowledgements:} We would like to thank Fernando Galaz Garc\'{\i}a, David Gonz\'alez \'Alvaro and Fred Wilhelm for useful comments on a preliminary version of this paper. 

 \section{Proof of Theorem \ref{thm:lower_bound}}
 \label{sec:lower_bound}
 
 In this section, we give a lower bound for the filling radius of an arbitrary closed Riemannian manifold in terms of its injectivity radius and an upper curvature bound.

 
  \begin{proof}[Proof of Theorem \ref{thm:lower_bound}]
 We will show that there is a continuous retraction of the tubular neighbourhood $\Phi:U_R(M)\to M$, by associating to each function $f\in U_R(M)$ the center of mass of a set in $M$ associated to $f$ (see \cite{grokar}, \cite{kar}, \cite{afsari}). To improve readability, we will denote the right hand side in \eqref{eq:fillrad_ineq} by $R_0$. 
 For each $R>0$, denote by $U_R(M)$ the open $R$-neighbourhood of $M$ inside $L^\infty(M)$, i.e.,
 \[
 U_R(M)=\left\{\,f\in L^\infty(M)\,:\, \dist(f,M)<R\,\right\}.
 \]
 Whenever $f\in U_R$, we define the \emph{vicinity set of $f$} as 
 \[
  A^R_f:=\left\{\,p\in M\,:\, \|f-\dist_p\|\leq R\,\right\}.
 \]
 It is clear from the definition that the sets $A^R_f$ are closed with nonempty interior, and for $R\leq R'$, there is an inclusion
$A^R_f \subset 	A^{R'}_f $; we also have that if $p,q\in A^R_f$, then 
 \[
\dist(p,q)=\norm{\dist_p-\dist_q}_{\infty}\leq 2R,
 \] 
 by the triangle inequality, thus $\diam A^R_f\leq 2R$.
 Furthermore, if $p\in A_f^R$, and $g\in L^\infty(M)$, then 
\[
\norm{\dist_p-g}_{\infty}\leq 
\norm{f-g}_\infty+\norm{\dist_p-f}_\infty,
\]
and therefore $p\in A_g^{R+\norm{g-f}_\infty}$. Interchanging the roles of $f$ and $g$, we get the sequence of inclusions
\[
A_f^R\subset A_g^{R+\norm{g-f}_\infty} \subset 
A_f^{R+2\norm{g-f}_\infty}. 
\]
Notice also that, as $\eps\to 0$, we get that 
\[
A_f^{R+\eps}\to A_f^R, \quad \text{that is, }
\bigcap_{\eps>0}A_f^{R+\eps} = A_f^R.
\]
 
 We will denote the characteristic function of a set $A$ as $\chi_A$ and the Riemannian measure of $M$ as $d\vol$.
 If we have a sequence $f_n\to f$ in $L^\infty(M)$, we have that, writing $\eps_n=\norm{f_n-f}$, we obtain 
 \[
 \int_M \,|\chi_{A_{f_n}^R}-\chi_{A_f^R}|\,d\vol
 \leq
 \int_M \,|\chi_{A_{f_n}^R}-\chi_{A_f^{R+\eps_n}}| \,d\vol
 +
 \int_M \,|\chi_{A_{f}^{R+\eps_n}}-\chi_{A_f^{R}}| \,d\vol,
 \]
 and thus
 \begin{equation}
     \label{eq:convergence_functions}
 \int_M \,|\chi_{A_{f_n}^R}-\chi_{A_f^R}|\,d\vol \to 0,
  \end{equation}
 as $n\to\infty$.
 
 
The set $A^R_f$ has nonempty interior, since it contains the set of points with $\norm{f-\dist_q}<R$, thus its volume in $M$ does not vanish. We can then consider the probability measure in $M$ defined as
 \[
 \nu_f=\frac{1}{\vol(A_f)}\,\chi_{A_f}\,d\vol,
 \]
 where $\vol$ is the Riemannian volume and $\chi_{A_f}$ is the characteristic function of the set $A_f$. Its support is $A_f$, and thus it has diameter less or equal than $2R$.  
 Observe that as $f_n\to f$ in $L^\infty$, the characteristic functions of $A_{f_n}$ converge to the characteristic function of $A_f$, and consequently, for any continuous function $g:M\to\RR$,
 \[
 \int_M\, g(x) d\nu_{f_n} \to  \int_M\, g(x) d\nu_{f}. 
 \]

Let $0<R<R_0$, where $R_0$ was defined  as the right hand side in inequality \eqref{eq:fillrad_ineq}. For any $f\in U_R$, its vicinity set is contained in a ball of radius $2R$, thus the main result in \cite[Theorem 2.1]{afsari} can be applied to the measure $\nu_f$ to obtain a unique point $p\in M$ characterized as the single minimizer of the function 
\[
F^f_2:M\to \RR, \qquad F_2(q):=\frac{1}{2}\int_M\, d^2(p,q)\,d\nu_f(p). 
\]
From equation \eqref{eq:convergence_functions}, it is clear that the assignment $\Phi:U_R(M)\to M$ mapping $f$ to $p$ is continuous. Moreover, when $f=\dist_p$, $A_f^R=B_R(p)$, and the minimum of $F^f_2$ agrees with $p$, thus $\Phi$ is a retraction, and 
\[
M\xrightarrow{ \ \iota_{R_0} \ } U_{R_0}(M)\xrightarrow{ \ \Phi \ } M
\]
is the identity map and  $(\Phi\circ \iota_{R_0})_\#[M]=[M]$. However, if $R_0>\fillrad(M)$, $(\iota_{R_0})_\#[M]=0$.
This finishes the proof.
 \end{proof}

	\section{Filling radius bounds of Riemannian submersions}
	Recall that a \emph{Riemannian submersion} is a submersion $\pi:M^{n+m}\to B^n$ between Riemannian manifolds such that for any point $p\in M$, the differential $d\pi_p$ preserves the length of vectors orthogonal to the fiber through $p$. 
	This implies that the fibers of $\pi$ give a partition of $M$ into equidistant sets, that is, for any $b_0,b_1\in B$, and any $z_0\in \pi^{-1}(b_0)$, the distance from $z_0$ to $\pi^{-1}(b_1)$ is equal to $\dist_B(b_0,b_1)$.  
	
	
Next, we give bounds on the filling radius of the total space of a Riemannian submersion, as stated in the introduction. 

	\begin{proof}[Proof of Theorem \ref{thm:submersion}]
	To facilitate the writing, denote by $\rho_0$ the right hand side in \eqref{eq:fillrad_submersion}
		As $\pi$ is a Riemannian submersion, we can isometrically embed $B\hookrightarrow L^{\infty}(M)$ with the map
		\begin{equation}
		    \varphi_1:B\to L^{\infty}(M), \qquad b\to\dist_b(\pi(\cdot)),
		\end{equation}
where $\dist_b(\pi(z)):=\dist_B(b,\pi(z))$, for all $z\in M$.

	Next, we translate this embedding by adding  $\rho_0$, that is, we have an isometric embedding
	\[
	\varphi:B\to L^{\infty}(M), \qquad b\to f_b,
	\] 
	where
	\[
    f_b:M\to\RR, \qquad	f_b(z):=\dist_b(\pi(z))+\rho_0.
	\]
	
	We will construct a deformation retract of $M$ onto the image $\varphi(B)$ adapting the main idea in 
\cite[Lemma 1]{katz} to our situation. The difference will be that, instead of constructing a cone over $M$ in a tubular neighbourhood of $M$, we will construct a \emph{mapping cylinder} of the Riemannian submersion $\pi:M\to B$.  For this, we define for each $t\in [0,D]$, and for each $p\in\pi^{-1}(b)$, the function 
	\[
		\dist_p^t(z)=
			\begin{cases}
		\min \ \{\dist_p(z)+t,f_b(z)\}, &\text{ if }\dist_p(z)<f_b(z)\\
		\max \, \{\dist_p(z)-t,f_b(z)\}, &\text{ if }\dist_p(z)\geq f_b(z). 
			\end{cases}
	\]
	We start by proving that $\norm{\dist_p^t-\dist_p}_{\infty}\leq t$. For this, we consider two cases:
		
		\begin{enumerate}
			\item When $\dist_p(z)<f_b(z)$, the function  $\dist_p$ is given by $\min \ \{\dist_p(z)+t,f_b(z)\}$; then we need to estimate 
			\[
			\max_{\{z\in M \, : \, \dist_p(z)<f_b(z)\}}  |\min \, \{\dist_p(z)+t,f_b(z)\}-\dist_p(z)|.
			\] 
			For this, we are going to use the sets: 
			\begin{alignat*}{3}
				&A \ &=& \ \{z  :  \dist_p(z)<f_b(z)\},\\
				&B \ &=& \ \{z  :  \dist_p(z)+t>f_b(z)\},\\
				&B^c \ &=& \ \{z  :  \dist_p(z)+t\leq f_b(z)\}.
			\end{alignat*}
				If $z\in A\cap B^c$, we have that 
		\[
		\max_{z\in A\cap B^c} \, |\dist_p^t(z)-\dist_p(z)|=\max_{z\in A\cap B^c} \, |\dist_p(z)+t-\dist_p(z)|=t.
		\]
		On the other hand, 
		\[
		A\cap B=\{z  :  \dist_p(z)<f_b(z)<\dist_p(z)+t\}=\{z  :  0<f_b(z)-\dist_p(z)<t\},
		\]
		 so, if $z\in A\cap B$, we have that 
		 \[
		 \max_{z\in A\cap B} \, |\dist_p^t(z)-\dist_p(z)|=\max_{z\in A\cap B} \, |f_b(z)-\dist_p(z)|<t.
		 \]
		Therefore, on $A$ we have obtained that 
		\[
		|\dist_p^t-\dist_p|\leq t.
		\]
		
		\item Suppose next that  $z\in A^c$, i.e., $\dist_p(z)\geq f_b(z)$. Now, we have to compute 
		\[
		\max_{\{z\in M \,  : \, \dist_p(z)\geq f_b(z)\}}  |\max \, \{\dist_p(z)-t,f_b(z)\}-\dist_p(z)|.
		\]
		In this case, we take the partition of $M$ given by the sets 
		\begin{alignat*}{2}
			&C \ &=& \ \{z  :  \dist_p(z)-t>f_b(z)\}.\\
			&C^c \ &=& \ \{z  :  \dist_p(t)-t\leq f_b(z)\}.
		\end{alignat*}
	If $z\in A^c\cap C$, we have that 
	\[
	\max_{z\in A^c\cap C} \, |\dist_p^t(z)-\dist_p(z)|=\max_{z\in A^c\cap C} \, |\dist_p(z)-t-\dist_p(z)|=t.
	\]
	On the other hand, 
	\[
	A^c\cap C^c=\{z  :  \dist_p(z)\geq f_b(z)\geq\dist_p(z)-t\}=\{z  :  0\leq\dist_p(z)-f_b(z)\leq t\},
	\]
	and 
	\[
	\max_{z\in A^c\cap C^c} \ |\dist_p^t(z)-\dist_p(z)|=\max_{z\in A^c\cap C^c}|f_b(z)-\dist_p(z)|\leq t.
	\]
    Thus, on $A^c$ we have that $|\dist_p^t-\dist_p|\leq t$, and combining it with the above, we have finally obtained that 
	\[
	\norm{\dist_p^t-\dist_p}_{\infty}\leq t.
	\]
		\end{enumerate}
		
		We prove next that for $\rho_0$ as defined as the upper bound in \eqref{eq:fillrad_submersion}, we have that 
		\[
		\dist_p^{\rho_0}=f_b, \qquad \text{where } b=\pi(p).
		\]
		Once again, we divide into cases:
		\begin{enumerate}
		    \item If $z\in A$, then our claim is clear, since for any Riemannian submersion, 
		    \[
		    \dist(p,z)\geq \dist_B(\pi(p),\pi(z)),
		    \]
		    and thus 
		    \[
		    f_b(z)=\dist_B(\pi(p),\pi(z))+\rho_0 \leq \dist(p,z)+\rho_0.
		    \]
		    \item On the other hand, for any $z\in M$, we have that there is at least one point $q$ in the fiber through $p$ with $\dist(z,q)=\dist_B(\pi(z),\pi(p))$, thus  (see Figure 1)
		    \[
		    \dist(p,z)\leq \dist(p,q)+\dist(q,z)\leq 
		    \diam \pi^{-1}(b) + \dist_B(\pi(z),\pi(p))\leq 
		    f_b(z)+\rho_0.
		    \]
		\end{enumerate}
			\begin{figure}[htbp]\centering\includegraphics[width=350pt]{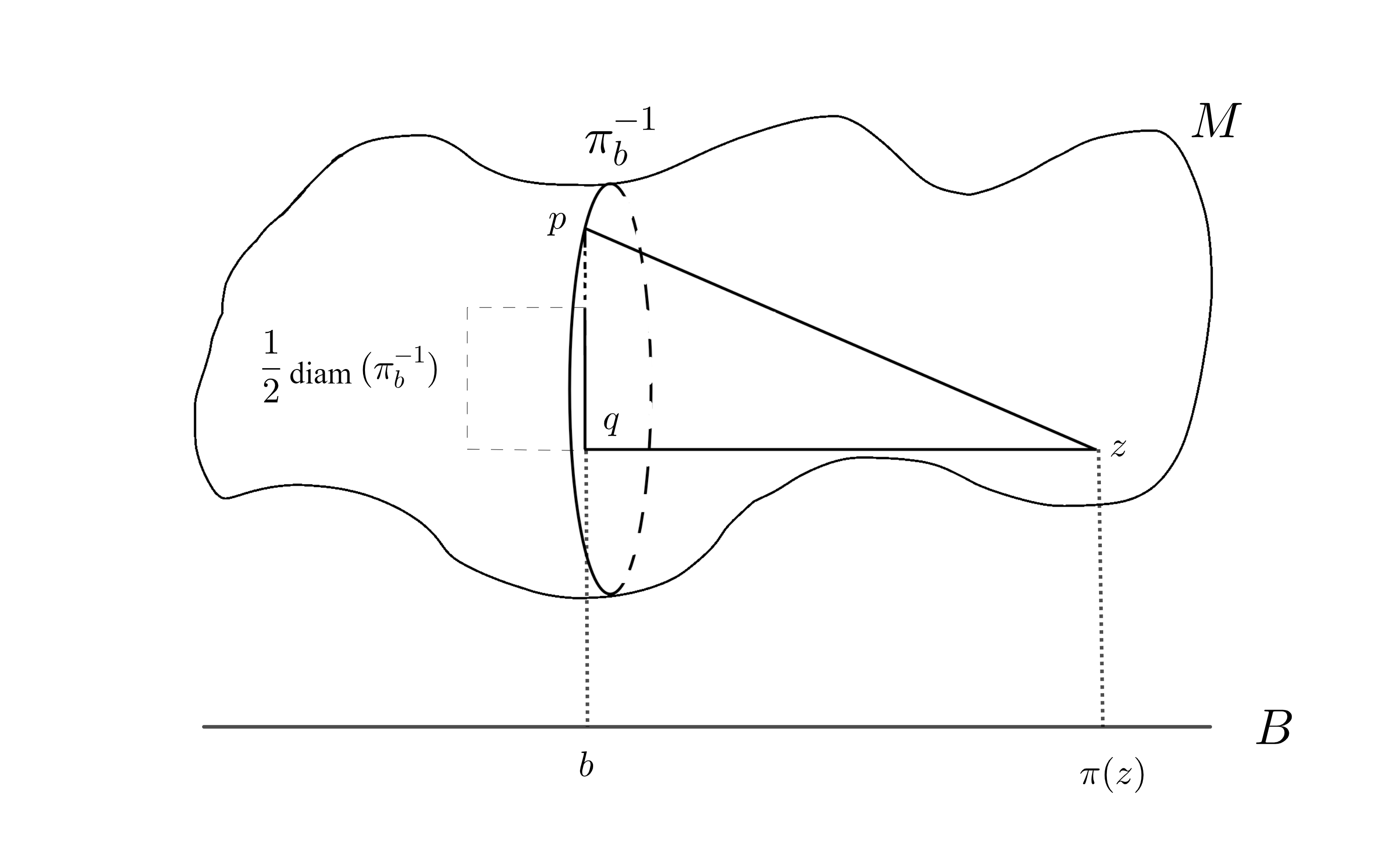}
			\caption{}
			\end{figure}
	
	Denoting by $\Cyl_\pi$  the cylinder map of $\pi:M\to B$,
	we have constructed an embedding of $\Cyl_\pi$ in the $\rho_0$-closed tubular neighbourhood of $M$; since $M$ retracts to $B$ in $\Cyl_\pi$, and $\mathrm{H}_n(B,\ZZ)=0$, we obtain that $\fillrad (M)\leq \rho_0$.
 	\end{proof}

There is a better estimate for Riemannian products, since Gromov proved in \cite{gromov} that the filling radius of a product satisfies
$\fillrad(M_1\times M_2)=\min(\fillrad (M_1),\fillrad (M_2))$. 
For warped products our Theorem provides: 
\begin{cor}
    For $B,F$ closed Riemannian manifolds, let $f:B\to (0,\infty)$ be a smooth function, and $M=B\times_f F$ the warped product over $B$ with fiber $F$. Then 
    \[
    \fillrad(M)\leq
    \min\{\,\fillrad(B), \frac12 \max f\cdot \diam F\,\}.
    \]
\end{cor}
\begin{proof}
$\fillrad(M)$ can not exceed half the maximum diameter of the fibers by the previous Theorem; by the definition of warped product, this explains the second term appearing in the above minimum. $\fillrad(B)$ appears by exactly the same explanation as in \cite[Pages 8-9]{gromov}
\end{proof}

\section{Filling radius of submetries}
\label{sec:submetry}
	
Recall that a submetry between metric spaces is a map $\pi:X\to B$ such that 
for every $p\in X$, any closed ball $B(p,r)$ of radius $r>0$ centred at $p$ maps onto the ball $B(\pi(p),r)$. They were introduced in \cite{berestovski}, as a purely metric version of Riemannian submersions. In fact, in \cite{guijaandbere}, Berestovskii and the first author proved that, when $X$ and $Y$ are Riemannian manifolds, submetries agree with $C^{1,1}$--Riemannian submersions.

As it is remarked on \cite{lim2020vietoris}, the definition of filling radius does not require the distance of $M$ to come from a  Riemannian metric. It suffices that $\widehat{\dist}_M$ generates the manifold topology. We call any $(M,\widehat{\dist}_M)$ satisfying this condition a \textit{metric manifold}.

When going over the proof of Theorem \ref{thm:submersion}, it is clear that the proof is entirely metric, and, except for the total space needing enough structure to have a fundamental class, the rest of the arguments carry verbatim to provide the following result: 

	\begin{cor}\label{cor:submetry}
		Let $(X, \widehat{\dist}_{X})$ be a metric manifold (i.e, a closed manifold with a distance), $(Y, \dist_{Y})$ a metric space and $\pi:X\to Y$ a submetry between them. Thus 
		\[
		\fillrad(X)\leq\frac12\max_{y\in Y}\{\diam\pi^{-1}(y)\}.
		\]
	\end{cor}

We should mention that, in the above corollary, $X$ does not need to be a manifold at all; for instance, $X$ can be replaced by a closed Alexandrov space, since such spaces have fundamental classes by the work of Yamaguchi in \cite{yama}.

		

This has some useful consequences that extend Theorem \ref{thm:submersion}:

\begin{cor}
    Suppose $M$ is a Riemannian manifold admitting a singular Riemannian foliation $\mathcal{F}$ with closed leaves. Then 
    \[
    \fillrad(M)\leq\frac12\max_{N\in \mathcal{F}} \, \{\diam N\}.
    \]
\end{cor}

As an example of the above, recall that when $G$ is a compact  Lie group acting by isometries on a closed Riemannian manifold $M$, the orbits of $G$ form a singular Riemannian foliation in $M$, thus
the filling radius of $M$ cannot exceed one half the diameter of the orbits.

\begin{example}
We can give another example (apart from the Berger spheres mentioned in the introduction) where the bound in Corollary \ref{cor:submetry} is better than the one third--$\diam$ bound. 
Let $M^n$ be a cohomogeneity one manifold with Grove-Ziller diagram $H\subset K=K'\subset G$ \cite{grozil}. Then $M^n$ can be given a cohomogeneity one metric containing a product cylinder $G/H\times [-\ell,\ell]$ with $\ell$ as large as desired. Its diameter will exceed $2\ell$, while the diameter of the fibers of $M/G$ remains uniformly bounded above.
\end{example}

\section{Reach of the Kuratowski embedding}\label{sec:reach}

We start by recalling the definition of the reach of a subset of a metric space. 

\begin{definition}[\cite{federer}]
	Let $(X,\dist)$ be a metric space and a $A\subset X$ a subset. Then, we define the set of points having a unique metric projection in $A$ as 
	\[
	\unp(A)=\{x\in X\colon\, \text{there exits a unique }a\in A\text{ such that }\dist(x,A)=\dist(x,a)\}.
	\]
	For $a\in A$, we define the \emph{reach} of $A$ at $a$, denoted as $\reach(a,A)$, as 
	\[
	\reach(a,A)=\sup\{r\geq0\colon B_r(a)\subset\unp(A)\}. 
	\]
	Finally, we say that $A$ has \textit{positive reach} if $\reach(a,A)>0$ for every $a\in A$.
\end{definition}

Intuitively, a set $A$ has positive reach if and only if there exists an open neighbourhood of it admitting a metric projection into $A$. When the ambient space is a Riemannian manifold, Lytchak proved in \cite[Theorem 1.1]{lytchak} that compact subsets $Z$ of positive reach have an upper curvature bound with respect to their inner metric. For an applied use of the reach, the reader can see \cite{harvey}.

	In \cite[Remark 4.2]{federer}, Federer also notes that $A\subseteq\unp(A)$, that $\reach(a,A)$ is continuous with respect to $a\in A$ and also 
		\[
		0\leq\reach(\partial A,a)\leq\reach(a,A)\leq\infty
		\]
		if $a\in\partial A$. Moreover, if $\reach(A)>0$, then $A$ is closed.
More interestingly, Federer notices that, for subsets of the Euclidean space, $\reach(A)=\infty$ if and only if $A$ is a convex closed set.
The main result of this section states that, in contrast to smooth isometric embeddings of manifolds in the Euclidean space, the reach of the image of a manifold under the Kuratowski embedding vanishes. 


\begin{proof}[Proof of Theorem \ref{thm:reach}]
	We need to prove that, given an arbitrary $\epsilon>0$, and $p\in M$, there is some function $f\in B_\epsilon(p)$ and a point $q\neq p\in M$ such that
	\begin{equation}
	{\label{propiedad}}
		\dist_{\infty}(f,M)=\norm{\dist_p-f}_{\infty}=\norm{\dist_q-f}_{\infty}.
	\end{equation}
	
	Let $0<\delta<<\epsilon$, and choose $q\in M$ such that 
	\[
	\norm{\dist_p-\dist_q}_{\infty}=\dist_M(p,q)=\delta.
	\]
	Define 
	\[
	f:M\to \RR, \qquad f(s):=\frac{1}{2}(\dist_p(s)+\dist_q(s)).
	\]
	We claim that this function does not belong to the Kuratowski embedding of $M$: otherwise, there would be some $r\in M$ with  
	\[
	f(s)=\frac{1}{2}(\dist_p(s)+\dist_q(s))=\dist_r(s),
	\]
	and evaluating at $s=r$, we would obtain that $\dist_p(r)=0=\dist_q(r)$, and $p$, $q$ and $r$ would all be the same point. 
	Also, $f$ lies in $B_\epsilon(p)$ since  
	\[
	\norm{f-\dist_p}_{\infty}=\norm{\frac{1}{2}(\dist_p+\dist_q)-\dist_p}_{\infty}=\frac{1}{2}\norm{\dist_q-\dist_p}_{\infty}=\frac{\delta}{2}<\epsilon.
	\]
	Moreover, a similar computation yields that 
	\[
	\norm{f-\dist_q}_{\infty}=\frac{\delta}{2}=\norm{f-\dist_p}_{\infty}.
	\]
	
	It only remains to show that there is not any point in $M$ closer to $f$. So assume that  there were some $\dist_r\in M$ such that 
	\[
	\norm{\dist_r-f}_{\infty}<\norm{\dist_p-f}_{\infty}.
	\] Then 
	\[
	\norm{\dist_r-f}_{\infty}<\norm{\dist_p-f}_{\infty}=\frac{\delta}{2}.
	\]
	Evaluating at $r$, we would obtain
	\begin{equation}{\label{uno}}
		|\dist_r(r)-f(r)|=|f(r)|=\frac12\left(\dist_p(r)+\dist_q(r)\right)<\frac{\delta}{2}.
	\end{equation}On the other hand, by the triangle inequality,
	\begin{equation}{\label{dos}}
		\dist_p(r)+\dist_q(r)\geq\dist_M(p,q)=\delta
	\end{equation} 
	resulting in a contradiction. 
	Thus $f$ satisfies \eqref{propiedad} and the proof is finished.
	
\end{proof}	 
 
 In spite of the last theorem, we should remark that for every point $p\in M$, and for any $\eps>0$, there are functions in $B_{\epsilon}(p)\subset L^\infty(M)$
 whose closest point in $M$ is unique and equal to $p$:
 translating  $f_{\delta}:=\dist_p+\delta$, it is easy to compute that  $\norm{f_\delta-d_p}_{\infty}=\delta$ and  every other point in $M$ is farther away.

\section{Intermediate Filling Radius}

If, instead of the fundamental class of the manifold, we look at a different fixed homology class in $M$, we arrive at the concept of the intermediate filling radius introduced in \cite{lim2020vietoris}. We adapt their definition to the case of $M\subset L^\infty(M)$. 
\begin{definition}[Intermediate Filling Radius]
	For any integer $k\geq 1$, any abelian group $\FF$, and any homology class $\omega\in\mathrm{H}_k(M;\FF)$, we define the \emph{$k$-intermediate filling radius of $\omega$} as 
	\[
	\fillrad_k(M,\FF,\omega):=\inf\{r>0 : \iota_{r,*}(\omega)=0\},
	\]
	where $\iota_r:M\hookrightarrow U_r(M)$ is the isometric embedding. This gives us a map 
	\[
	\fillrad_k(M, \FF, \cdot):\mathrm{H}_k(M;\FF)\to\mathbb{R}_{\geq0}.
	\] 
	Finally, the \emph{intermediate $k$-filling radius of $M$} is the infimum of the function $\fillrad_k(M, \FF, \cdot)$ over $\mathrm{H}_k(M;\FF)$.
\end{definition}

We will usually omit $\FF$ from the notation if it does not create confusion. Our first observation is to restrict the possible values of $\fillrad_k$. 
\begin{prop}
For $M$ as in Theorem \ref{thm:lower_bound}, we have that 
\begin{equation}
\label{eq:fillrad_k_lower_bound}
	\fillrad_k(M) \geq  \frac14\left\{\,\inj M\,,\,\frac{\pi}{\sqrt{\Delta}}\,\right\}.
\end{equation}
\end{prop}
\begin{proof}
For $R$ smaller than the right hand side in $\eqref{eq:fillrad_k_lower_bound}$, we have a retraction of $U_R(M)$ onto $M$, as proven in Theorem $\ref{thm:lower_bound}$. 
\end{proof}

The following result can be considered a mild extension of \cite[Theorem 1.1]{liu} to this intermediate invariant. 

\begin{thm}
{\label{intermedios}}
	Let $f:M^m\to N^n$ be a Lipschitz map between closed manifolds such that  the induced map  $f_{k,*}:\mathrm{H}_k(M)\to\mathrm{H}_k(N)$ is  onto. Then \[\fillrad_k(M)\geq C^{-1}\fillrad_k(N),\]
	where 
	\[
	C=\dil f=\sup_{p\neq p'}\frac{\dist_N(f(p),f(p'))}{\dist_M(p,p')}.
	\]
\end{thm}

\begin{proof}
As in \cite{liu}, we extend the map $f:M\to N\to L^\infty(N)$ to a map $\widetilde{f}:L^\infty(M)\to L^\infty(N)$ that remains Lipschitz with the same dilation as $f$. Then 
\[
\tildef(U_R(M))\subset U_{C\cdot R}(N),
\]
and we have a commutative diagram in homology
        \[
\xymatrix{\mathrm{H}_k(M)\ar[d]_{\iota_{*}}\ar[r]^{f_{k,*}} &\mathrm{H}_k(N)\ar[d]^{\iota_{*}}\\
		\mathrm{H}_k(U_{R}(M))\ar[r]^{\widetilde{f}_{k,*}}&\mathrm{H}_k(U_{C\cdot R}(N))}
		\]
	If $C\cdot R<\fillrad_k(N)$, there is some $ a\in\mathrm{H}_k(N)$ such that $\iota_{*}(a)\neq 0$, hence there is some $b\in\mathrm{H}_k(M)$ such that 
	$\widetilde{f}_{k,*}\iota_{*}(b)\neq 0$; then  $\iota_{*}(b)\neq 0$, and $R<\fillrad_k(M)$.
	
	
\end{proof}


	

\nocite{*}
\printbibliography

\end{document}